\documentclass{article}%
\usepackage{amsmath}
\usepackage{amssymb}
\usepackage{amsfonts}
\usepackage{graphicx}%
\providecommand{\U}[1]{\protect\rule{.1in}{.1in}}
\newtheorem{theorem}{Theorem}[section]
\newtheorem{acknowledgement}[theorem]{Acknowledgement}

\newtheorem{corollary}[theorem]{Corollary}

\newtheorem{definition}[theorem]{Definition}
\newtheorem{example}[theorem]{Example}

\newtheorem{proposition}[theorem]{Proposition}
\newtheorem{remark}[theorem]{Remark}

\newenvironment{proof}[1][Proof]{\noindent\textbf{#1.} }{\ \rule{0.5em}{0.5em}}
\begin{document}

\title{The multiplicative ideal theory of Leavitt path algebras of directed graphs
\ - a survey}
\author{Kulumani M. Rangaswamy\\Department of Mathematics, University of Colorado\\Colorado Springs, Colorado 80918}
\date{}
\maketitle

\begin{abstract}
Let $L$ be the Leavitt path algebra of an arbitrary directed graph $E$ over a
field $K$. This survey article describes how this highly non-commutative ring
$L$ shares a number of the characterizing properties of a Dedekind domain or a
Pr\"{u}fer domain expressed in terms of their ideal lattices. Special types of
ideals such as the prime, the primary, the irreducible and the radical ideals
of $L$ are described by means of the graphical properties of $E$. The
existence and the uniqueness of the factorization of a non-zero ideal of $L$
as an irredundant product of prime or primary or irreducible ideals is
established. Such factorization always exists for every ideal in $L$ if the
graph $E$ is finite or if $L$ is two-sided artinian or two-sided noetherian.
In all these factorizations, the graded ideals of $L$ seem to play an
important role. Necessary and sufficient conditions are given under which $L$
is a generalized ZPI ring, that is, when every ideal of $L$ is a product of
prime ideals. Intersections of various special types of ideals are
investigated and an anlogue of Krull's theorem on the intersection of powers
of an ideal in $L$ is established.

\end{abstract}

\section{Introduction}

Leavitt path algebras of directed graphs are algebraic analogues of graph
C*-algebras and, ever since they were introduced in 2004, have become an
active area of research. Every Leavitt path algebra $L:=L_{K}(E)$ of a
directed graph $E$ over a field $K$ is equipped with three mutually compatible
structures: $L$ is an associative $K$-algebra, $L$ is a $%
\mathbb{Z}
$-graded algebra and $L$ is an algebra with an involution $^{\ast}$. Further,
$L$ possesses a large supply of idempotents, but it is highly non-commutative.
Indeed, in most of the cases, the center of this $K$-algebra is trivial, being
just the field $K$. In spite of this, it is somewhat intriguing and certainly
interesting that the ideals of such a non-commutative algebra $L$ exhibit the
behavior of the ideals of a Pr\"{u}fer domain and sometimes that of a Dedekind
domain, thus making the multiplicative ideal theory of these algebras $L$
worth investigating. The purpose of this survey is to give a detailed account
of some of these properties of $L$ and the resulting factorizations of its
ideals. To start with, the ideal multiplication \ in $L$ is commutative:
$AB=BA$ for any two ideals $A,B$ of $L$. As we shall see, the
Prufer-domain-like properties of $L$ lead to satisfactory factorizations of
ideals of $L$ as products of prime, primary or irreducible ideals. The graded
ideals of $L$ seem to possess interesting properties such as, coinciding with
their own radical, being realizable as Leavitt path algebras of suitable
graphs, possessing local units and many others. They play an important role in
the factorization of non-graded ideals of $L$. As noted in (\cite{AAS},
Theorem 2.8.10 and in \cite{R2}), the two-sided ideal structure of $L$ can be
described completely in terms of the hereditary saturated subsets and breaking
vertices and cycles without exits in the graph $E$ and irreducible polynomials
in $K[x,x^{-1}]$, and the association preserves the lattice structures. This
fact facilitates the description various factorization properties of the
two-sided ideals in $L$.

This paper is organized as follows. After the Preliminaries, Section 3
describes the various properties of the graded ideals of $L$ which are
foundational to the study of non-graded ideals and in the factorization of
ideals in $L$. In Section 4, $L$ is shown to be an arithmetical ring, that is,
its ideal lattice is distributive and, as a consequence, the Chinese Remainder
Theorem holds in $L$. In addition, $L$ is shown to be a multiplication ring.
The ideal version of the number-theoretic theorem $\gcd(m,n)\cdot
\operatorname{lcm}(m,n)=mn$ for positive integers $m,n$ holds in $L$, namely,
for any two ideals $M,N$ in $L$, $(M\cap N)(M+N)=MN$, again a characterizing
property of Pr\"{u}fer domains. In the next section, the prime, the primary,
the irreducible and the radical ideals of $L$ are described in terms of the
graph properties of $E$. It is interesting to note that for a graded ideal $I$
of $L$ the first three of these properties coincide and that $I$ is always a
radical ideal. In Section 6, we consider the existence and the uniqueness of
factorizations of a non-zero ideal $I$ as a product of prime, primary or
irreducible ideals of $L$. It is shown that if $E$ is a finite graph or more
generally, if $L$ is two-sided noetherian or artinian, then every ideal of $L$
is a product of prime ideals. This leads to a complete characterization of $L$
as a generalized ZPI ring, that is, a ring in which every ideal of $L$ is a
product of prime ideals. Finally, an analogue of the Krull's theorem on powers
of an ideal is proved for Leavitt path algebras. The results of this paper
indicate the potential for successful utilization of the ideas and results
from the ideal theory of commutative rings in the deeper study of the ideal
theory of Leavitt path algebras (of course using different techniques, as $L$
is non-commutative, and using the graphical properties of $E$ and the nature
of the graded ideals of $L$).

\section{Preliminaries}

For the general notation, terminology and results in Leavitt path algebras, we
refer to \cite{AAS}, \cite{R1} and \cite{T} , and for those in graded rings,
we refer to \cite{H}, \cite{NO}. We refer to \cite{FHL} - \cite{HK} for
results in commutative rings. Below we give an outline of some of the needed
basic concepts and results.

A (directed) graph $E=(E^{0},E^{1},r,s)$ consists of two sets $E^{0}$ and
$E^{1}$ together with maps $r,s:E^{1}\rightarrow E^{0}$. The elements of
$E^{0}$ are called \textit{vertices} and the elements of $E^{1}$
\textit{edges}. For each $e\in E^{1}$, say,%
\[%
\begin{array}
[c]{ccc}%
\bullet_{s(e)} & \overset{e}{\longrightarrow} & \bullet_{r(e)}%
\end{array}
\]
$s(e)$ is called the \textbf{source} of $e$ and $r(e)$ the \textbf{range} of
$e$. If $\overset{u}{\bullet}$ $\overset{e}{\longrightarrow}$
$\overset{v}{\bullet}$ is an edge, then $\overset{u}{\bullet}$
$\overset{e^{\ast}}{\longleftarrow}$ $\overset{v}{\bullet}$ denotes the
\textbf{ghost edge} $e^{\ast}$ with $s(e^{\ast})=v$ and $r(e^{\ast})=u$.

A vertex $v$ is called a \textbf{sink }if it emits no edges and a vertex $v$
is called a \textbf{regular vertex} if it emits a non-empty finite set of
edges. An \textbf{infinite emitter} is a vertex which emits infinitely many edges.

A \textbf{path} $\mu$ of length $n$ is a sequences of edges $\mu$ $=e_{1}\dots
e_{n}$ where $r(e_{i})=s(e_{i+1})$ for all $i=1,\cdot\cdot\cdot,n-1$. $|\mu|$
denotes the length of $\mu$. The path $\mu$ $=e_{1}\dots e_{n}$ in $E$ is
\textbf{closed} if $r(e_{n})=s(e_{1})$, in which case $\mu$ is said to be
\textit{based at the vertex }$s(e_{1})$. A closed path $\mu$ as above is
called \textbf{simple} provided it does not pass through its base more than
once, i.e., $s(e_{i})\neq s(e_{1})$ for all $i=2,...,n$. The closed path $\mu$
is called a \textbf{cycle} if it does not pass through any of its vertices
twice, that is, if $s(e_{i})\neq s(e_{j})$ for every $i\neq j$.

An \textit{exit }for a path $\mu=e_{1}\dots e_{n}$ is an edge $e$ such that
$s(e)=s(e_{i})$ for some $i$ and $e\neq e_{i}$.

If there is a path from vertex $u$ to a vertex $v$, we write $u\geq v$. A
subset $D$ of vertices is said to be \textbf{downward directed }\ if for any
$u,v\in D$, there exists a $w\in D$ such that $u\geq w$ and $v\geq w$. A
subset $H$ of $E^{0}$ is called\textbf{\ hereditary} if, whenever $v\in H$ and
$w\in E^{0}$ satisfy $v\geq w$, then $w\in H$. A hereditary set is
\textbf{saturated }if, for any regular vertex $v$, $r(s^{-1}(v))\subseteq H$
implies $v\in H$.

\begin{definition}
Given an arbitrary graph $E$ and a field $K$, the \textit{Leavitt path algebra
}$L_{K}(E)$ is defined to be the $K$-algebra generated by a set $\{v:v\in
E^{0}\}$ of pair-wise orthogonal idempotents, together with a set of variables
$\{e,e^{\ast}:e\in E^{1}\}$ which satisfy the following conditions:

(1) \ $s(e)e=e=er(e)$ for all $e\in E^{1}$.

(2) $r(e)e^{\ast}=e^{\ast}=e^{\ast}s(e)$\ for all $e\in E^{1}$.

(3) (The "CK-1 relations") For all $e,f\in E^{1}$, $e^{\ast}e=r(e)$ and
$e^{\ast}f=0$ if $e\neq f$.

(4) (The "CK-2 relations") For every regular vertex $v\in E^{0}$,
\[
v=\sum_{e\in E^{1},s(e)=v}ee^{\ast}.
\]

\end{definition}

Note that $L$ need not have an identity. Indeed, $L$ will have the identity
$1$ exactly when the vertex set $E^{0}$ is finite and in that case $1=%
{\displaystyle\sum\limits_{v\in E^{0}}}
v$. However, $L$ possesses \textbf{local units}, namely, given any finite set
of elements elements $a_{1},\cdot\cdot\cdot,a_{n}\in L$, there is an
idempotent $u$ such that $ua_{i}=a_{i}=a_{i}u$ for all $i=1,\cdot\cdot\cdot
,n$. Every element $a\in L:=L_{K}(E)$ can be written as $a=$ $%
{\displaystyle\sum\limits_{i=1}^{n}}
k_{i}\alpha_{i}\beta_{i}^{\ast}$ where $\alpha_{i},\beta_{i}$ are paths and
$k_{i}\in K$. Here $r(\alpha_{i})=s(\beta_{i}^{\ast})=r(\beta_{i})$. From
this, it is easy to see that $L=%
{\displaystyle\bigoplus\limits_{u\in E^{0}}}
Lu$.

Many well-known examples of rings occur as Leavitt path algebras.

\begin{example}
The Leavitt path algebra of the straight line graph $E:$
\[
\bullet_{v_{1}}\overset{e_{1}}{\longrightarrow}\bullet_{v_{2}}\overset{e_{2}%
}{\longrightarrow}\cdot\cdot\cdot\overset{e_{n-1}}{\longrightarrow}%
\bullet_{v_{n}}%
\]
is isomorphic to the matrix ring $M_{n}(K)$.
\end{example}

(Indeed, if $p_{1}=e_{1}\cdot\cdot\cdot e_{n-1}$, $p_{2}=e_{2}\cdot\cdot\cdot
e_{n-1}$, $\cdot\cdot\cdot$ , $p_{n-1}=e_{n-1}$, $p_{n}=v_{n}$, then
$\{\epsilon_{ij}=p_{i}p_{j}^{\ast}:1\leq i,j\leq n\}$ is a set of matrix
units, that is, $\epsilon_{ii}^{2}=\epsilon_{ii}$ and $\epsilon_{ij}%
\epsilon_{jk}=\epsilon_{ik}$. Then $\epsilon_{ij}\longmapsto E_{ij}$ induces
the isomorphism, where $E_{ij}$ is the $n\times n$ matrix with $1$ at $(i,j)$
position and $0$ everywhere else. )

\bigskip

\begin{example}
\label{Laurent} If $E$ is the graph with a single vertex and a single loop
\[%
\raisebox{-0pt}{\includegraphics[
natheight=0.396100in,
natwidth=0.989300in,
height=0.4255in,
width=1.0222in
]%
{Loop.jpg}%
}
\]
then $L_{K}(E)\cong K[x,x^{-1}]$, the Laurent polynomial ring, induced by the
map $v\mapsto1$, $x\mapsto x$, $x^{\ast}\mapsto x^{-1}$.
\end{example}

The defining relations of a Leavitt path algebra $L_{K}(E)$ show that it is a
non-commutative ring. Indeed if $e$ is an edge in $E$, say,
$\overset{u}{\bullet}$ $\overset{e}{\longrightarrow}$ $\overset{v}{\bullet}$
$\ $\ where $u\neq v$, then by defining relation (1), $ue=e$, but
$eu=evu=e(vu)=0$. The following Proposition describes when $L_{K}(E)$
\ becomes a commutative ring.

\begin{proposition}
Let $E$ be a connected graph. Then the Leavitt path algebra $L_{K}(E)$ is
commutative if and only if either $E$ consists of just a single vertex
$\{\bullet\}$ or $E$ is the graph with a single vertex and a single loop as in
Example 2 above. In this case $L_{K}(E)$ $\cong K$ or $K[x,x^{-1}]$.
\end{proposition}

Every Leavitt path algebra $L_{K}(E)$ is a $%
\mathbb{Z}
$\textbf{-graded algebra}, namely, $L_{K}(E)=%
{\displaystyle\bigoplus\limits_{n\in\mathbb{Z}}}
L_{n}$ induced by defining, for all $v\in E^{0}$ and $e\in E^{1}$, $\deg
(v)=0$, $\deg(e)=1$, $\deg(e^{\ast})=-1$. Here the $L_{n}$, called
\textbf{homogeneous components}, are abelian subgroups satisfying $L_{m}%
L_{n}\subseteq L_{m+n}$ for all $m,n\in%
\mathbb{Z}
$. Further, for each $n\in%
\mathbb{Z}
$, the subgroup $L_{n}$ is given by
\[
L_{n}=\{%
{\textstyle\sum}
k_{i}\alpha_{i}\beta_{i}^{\ast}\in L:\text{ }|\alpha_{i}|-|\beta_{i}|=n\}.
\]
An ideal $I$ of $L_{K}(E)$ is said to be a \textbf{graded ideal} if $I=$ $%
{\displaystyle\bigoplus\limits_{n\in\mathbb{Z}}}
(I\cap L_{n})$. If $I$ is a non-graded ideal, then $%
{\displaystyle\bigoplus\limits_{n\in\mathbb{Z}}}
(I\cap L_{n})$ is the largest graded ideal contained in $I$ \ and is called
the \textbf{graded part} of $I$, denoted by $gr(A)$.

We will also be using the fact that the Jacobson radical (and in particular,
the prime/Baer radical) of $L_{K}(E)$ is always zero (see \cite{AAS}).

Let $\Lambda$ be an arbitrary non-empty (possibly, infinite) index set. For
any ring $R$, we denote by $M_{\Lambda}(R)$ the ring of matrices over $R$
whose entries are indexed by $\Lambda\times\Lambda$ and whose entries, except
for possibly a finite number, are all zero. It follows from the works in
\cite{AM} that $M_{\Lambda}(R)$ is Morita equivalent to $R$.

\textbf{Throughout this paper }$L$\textbf{ will denote the Leavitt path
algebra }$L_{K}(E)$\textbf{ of an arbitrary directed graph }$E$\textbf{ over a
field }$K$\textbf{.}

\section{Graded ideals of a Leavitt path algebra}

In this section, we shall describe some of the salient properties of the
graded ideals of a Leavitt path algebra $L$. As we shall see in a later
section, these properties impact the factorization of ideals of $L$. Every
ideal of $L$, whether graded or not, is shown to possess an orthogonal set of
generators. As a consequence, we get the interesting property that every
finitely generated ideal of $L$ is a principal ideal. It is interesting to
note that if $I$ is a graded ideal of $L$, then both $I$ and $L/I$ can be
realized as Leavitt path algebras of suitable graphs.

Suppose $H$ is a hereditary saturated subset of vertices. A \textbf{breaking
vertex }of $H$ is an infinite emitter $w\in E^{0}\backslash H$ with the
property that $0<|s^{-1}(w)\cap r^{-1}(E^{0}\backslash H)|<\infty$. The set of
all breaking vertices of $H$ is denoted by $B_{H}$. For any $v\in B_{H}$,
$v^{H}$ denotes the element $v-\sum_{s(e)=v,r(e)\notin H}ee^{\ast}$. The
following theorem of Tomforde describes graded ideals of $L$ by means of their generators.

\begin{theorem}
(\cite{T}) \label{graded Ideal} Suppose $H$ is a hereditary saturated set of
vertices and $S$ is a subset of $B_{H}$. Then the ideal $I(H,S)$ generated by
the set of idempotents $H\cup\{v^{H}:v\in S\}$ is a graded ideal of $L$ and
conversely, every graded ideal $I$ of $L$ is of the form $I(H,S)$ where
$H=I\cap E^{0}$ and $S=\{u\in B_{H}:u^{H}\in I\}$.
\end{theorem}

Given a pair $(H,S)$ where $H$ is a hereditary saturated set of vertices in
the graph $E$ and $S$ is a subset of $B_{H}$, one could construct the
\textbf{Quotient graph} $E\backslash(H,S)$ given by $(E\backslash
(H,S))^{0}=E^{0}\backslash H\cup\{u^{\prime}:u\in B_{H}\backslash S\}$,
$(E\backslash(H,S))^{1}=\{e\in E^{1}:r(e)\notin H\}\cup\{e^{\prime}:e\in
E^{1}$ with $r(e)\in B_{H}\backslash S\}$ and $r,s$ are extended to
$(E\backslash(H,S))^{0}$ by setting $s(e^{\prime})=s(e)$ and $r(e^{\prime
})=r(e)^{\prime}$.

The next theorem describes a generating set \ $Y$ for a not necessarily-graded
non-zero ideal of $L$ . This set $Y$ is actually an orthogonal set of generators.

\begin{theorem}
(\cite{R2}) Let $E$ be an arbitrary graph and let $I$ be an arbitrary non-zero
ideal of $L=L_{K}(E)$ with $H=I\cap E^{0}$ and $S=\{u\in B_{H}:u^{H}\in I\}$.
Then $I$ is generated by the set%
\[
Y=H\cup\{v^{H}:v\in S\}\cup\{f_{t}(c_{t}):t\in T\}
\]
where $T$ is some index set (which may be empty), for each $t\in T$, $c_{t}$
is a cycle without exits in $E\backslash(H,S)$, no vertex $v\in S$ lies on any
cycle $c_{t}$, $t\in T$ and $f_{t}(x)\in K[x]$ is a polynomial with a non-zero
constant term and is of the smallest degree such that $f_{t}(c_{t})\in I$. Any
two elements $x\neq y$ in $Y$ are orthogonal, that is, $xy=0=yx$.
\end{theorem}

If $I$ is a finitely generated ideal, then the orthogonal set $Y$ of
generators mentioned in the above theorem can be shown to be finite and, in
that case, the single element $a=%
{\displaystyle\sum\limits_{y\in Y}}
y$ will be a generator for the ideal $I$. Consequently, \ we obtain the
following interesting result.

\begin{theorem}
\label{Fin. gen. principal}(\cite{R2}) Every finitely generated ideal in a
Leavitt path algebra is a principal ideal.
\end{theorem}

\begin{remark}
In \cite{AMT}, the above theorem has been extended by showing that every
finitely generated one-sided ideal of $L$ is a principal ideal, that is, $L$
is a B\^{e}zout ring.
\end{remark}

An important property of graded ideals is the following.

\begin{theorem}
(\cite{RT}) \ Every graded ideal $I(H,S)$ of $L$ can be realized as a Leavitt
path algebra $L_{K}(F)$ of some graph $F$ and further\ the corresponding
quotient ring $L/I(H,S)$ is also a Leavitt path algebra, being isomorphic to
the Leavitt path algebra $L_{K}(E\backslash(H,S))$ of the quotient graph
$E\backslash(H,S)$.
\end{theorem}

Since Leavitt path algebras possess local units, we conclude that the graded
ideals $I$ of $L$ possess local units. Using this, we obtain some interesting
properties of graded ideals.

\begin{proposition}
\label{graded Ideal Property} (\cite{R3}) (i) Let $A$ be a graded ideal of
$L$. Then

(a) for any ideal $B$ of $L$, $AB=A\cap B$, $BA=B\cap A$ and, in particular,
$A^{2}=A$;

(b) $AB=BA$ for all ideals $B$;

(c) If $A=A_{1}\cdot\cdot\cdot A_{m}$ is a product of ideals, then $A=%
{\displaystyle\bigcap\limits_{i=1}^{m}}
gr(A_{i})=$ $%
{\displaystyle\prod\limits_{i=1}^{m}}
gr(A_{i})$. Similarly, if $A=A_{1}\cap\cdot\cdot\cdot\cap A_{m}$ is an
intersection of ideals $A_{i}$, then $A=%
{\displaystyle\bigcap\limits_{i=1}^{m}}
gr(A_{i})=$ $%
{\displaystyle\prod\limits_{i=1}^{m}}
gr(A_{i})$.

(ii) If $A_{1},\cdot\cdot\cdot,A_{m}$ are graded ideals of $L$, then $%
{\displaystyle\prod\limits_{i=1}^{m}}
A_{i}=%
{\displaystyle\bigcap\limits_{i=1}^{m}}
A_{i}$.
\end{proposition}

\begin{proof}
We shall point out the easy proof of (i)(a). We need only to prove $A\cap
B\subseteqq AB$. Let $x\in A\cap B$. Since the graded ideal $A$ has local
units, there is an idempotent $u\in A$ such that $ua=a=au$. Clearly then
$a=ua\in AB$. So $A\cap B=AB$. Similarly, $B\cap A=BA$. Hence $AB=BA$. In
particular, $A^{2}=A\cap A=A$.
\end{proof}

A natural question is when every ideal of $L$ will be a graded ideal. This can
happen when $E$ satisfies the following graph property.

\begin{definition}
A graph $E$ satisfies \textbf{Condition (K)} if whenever a vertex $v$ lies on
a simple closed path $\alpha$, $v$ also lies on another simple closed path
$\beta$ distinct from $\alpha$.
\end{definition}

Here is a simple graph satisfying Condition (K), where every vertex satisfies
the required property.%

\[%
\begin{array}
[c]{ccccc}%
\bullet & \longleftarrow & \bullet & \longleftarrow & \bullet\\
& \searrow &  & \nearrow & \\
&  & \bullet &  & \\
& \nearrow &  & \searrow & \\
\bullet_{v} & \longleftarrow & \bullet & \longleftarrow & \bullet
\end{array}
\]

\begin{theorem}
(\cite{R1}, \cite{T}) The following conditions are equivalent for
$L=:L_{K}(E)$:

(a) Every ideal of $L$ is graded;

(b) Every prime ideal of $L$ is graded;

(c) The graph $E$ satisfies Condition (K).
\end{theorem}

\section{The lattice of ideals of a Leavitt path algebra}

This section describes how the ideals of a Leavitt path algebra $L$ share
lattice-theoretic properties and module-theoretic properties of the ideals of
a Dedekind domain or a Pr\"{u}fer domain. We start with noting that, in this
non-commutative ring $L$, the multiplication of ideals is commutative.
\ Moreover, $L$ is left/right hereditary, that is, every left/right or
two-sided ideal of $L$ is projective as a left or a right ideal. The ideal
lattice of $L$ is distributive and multiplicative. It is also shown how many
of the characterizing properties of a Pr\"{u}fer domain stated in terms of its
ideals hold in $L$.

Using a deep theorem of George Bergman, Ara and Goodearl proved the following
result that every Leavitt path algebra is a left/right hereditary ring, a
property shared by Dedekind domains.

\begin{theorem}
(Theorem 3.7, \cite{AG}) Every ideal (including any one-sided ideal) of a
Leavitt path algebra $L$ is projective as a left/right $L$-module.
\end{theorem}

In Section 3, we noted that if $A$ is a graded ideal of $L$, then $AB=BA$ for
any ideal $B$ of $L$. What happens if $A$ is not a graded ideal ? With an
analysis of the "non-graded parts" of $A$ and $B$, it was shown in \cite{AAS}
and \cite{R3} that even though $L$ is, in general, non-commutative, the
multiplication of its ideals is commutative as noted next.

\begin{theorem}
(\cite{AAS}, \cite{R3}) For any two arbitrary ideals $A,B$ of a Leavitt path
algebra $L$, $AB=BA$.
\end{theorem}

The next result shows that every Leavitt path algebra $L$ is an arithmetical
ring, that is the ideal lattice of $L$ is distributive, a property that
characterizes Pr\"{u}fer domains.

\begin{theorem}
\label{LPA is Arithmetic} (\cite{R3}) For any three ideals $A,B,C$ of the
Leavitt path algebra $L$, we have
\[
A\cap(B+C)=(A\cap B)+(A\cap C).
\]

\end{theorem}

\begin{remark}
A well-known result in commutative rings (see e.g. Theorem 18, Ch. V,
\cite{ZS}) states that if the ideal lattice of a commutative ring $R$ is
distributive (such as when $R$ is a Dedekind domain), then the Chinese
Remainder Theorem holds in $R$: This means that the simultaneous congruences
$x\equiv x_{i}$ $(\operatorname{mod}$ $A_{i})$ ($i=1,\cdot\cdot\cdot,n$) where
the $A_{i}$ are ideals and the elements $x_{i}\in R$, admits a solution for
$x$ in $R$ provided the compatibility condition $x_{i}+x_{j}\equiv0$
$(\operatorname{mod}A_{i}+A_{j})$ holds for all $i\neq j$. The proof of this
theorem does not require $R$ to be commutative and nor does it require the
existence of a multiplicative identity in $R$. So, as a consequence of Theorem
\ref{LPA is Arithmetic}, one can show that the Chinese Remainder Theorem holds
in Leavitt path algebras. (Thus Leavitt path algebras satisfy another property
of Dedekind domains).
\end{remark}

We next use Theorem \ref{LPA is Arithmetic}, to show that every Leavitt path
algebra is a multiplication ring, a useful property in the multiplicative
ideal theory of Leavitt path algebras.

\begin{theorem}
\label{LPAs are Multiplication rings} (\cite{R3}) The Leavitt path algebra
$L=L_{K}(E)$ of an arbitrary graph $E$ is a multiplication ring, that is, for
any two ideals $A,B$ of $L$ with $A\subseteq B$, there is an ideal $C$ of $L$,
such that $A=BC=CB$. Moreover, if $A$ is a prime ideal, then $AB=A=BA$.
\end{theorem}

A well-known property of a Dedekind domain $R$ is that if there are only
finitely many prime ideals in $R$, then $R$ is a principal ideal domain (see
Theorem 16, Ch. V in \cite{ZS}). Interestingly, as the next theorem shows, a
Leavitt path algebra possesses this property.

\begin{theorem}
\label{Finitely many primes} (\cite{EER}) Let $L:=L_{K}(E)$ be the Leavitt
path algebra of an arbitrary graph $E$. If $L$ has only a finite number of
prime ideals, then every ideal of $L$ is a principal ideal.
\end{theorem}

This follows from Proposition 2.10 of \cite{EER} that $L$ then has only
finitely many ideals and so they satisfy the ascending chain condition.
Consequently they are all finitely generated. By Theorem
\ref{Fin. gen. principal}, they are principal ideals.

In a recent paper \cite{EEKR}, it has been shown that the ideals of a Leavitt
path algebra satisfy two more characterizing properties of Pr\"{u}fer domains
among integral domains.

\begin{theorem}
(\cite{EEKR}) Let $A,B,C$ be any three ideals of a Leavitt path algebra $L$. Then

(i) $\ A(B\cap C)=AB\cap AC$;

(ii) $(A\cap B)(A+B)=AB$.
\end{theorem}

Note that the statements (i) and (ii) in the preceding theorem are the ideal
versions of well-known theorems in elementary number theory, namely, for any
three positive integers $a,b,c$, we have $a\cdot\gcd(b,c)=\gcd(ab,ac)$ and
$\gcd(a,b)\cdot\operatorname{lcm}(a,b)=ab$.

However, not all the characterizing properties of a Pr\"{u}fer domain hold in
a Leavitt path algebra. For instance, a domain $R$ is a Pr\"{u}fer domain if
and only if non-zero finitely generated ideals of $R$ are cancellative, that
is, if $A$ is a non-zero finitely generated ideal, then for any two ideals
$B,C$ of $R$, $AB=AC$ implies $B=C$. \ This property may not hold in a Leavitt
path algebra as the next example shows:

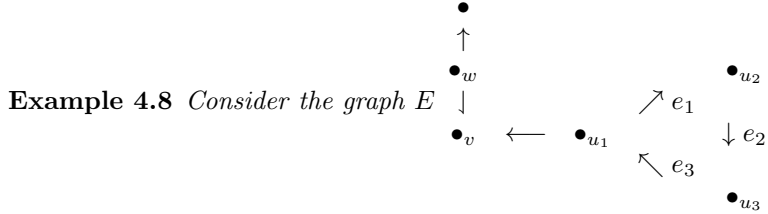
\begin{example}
Consider the graph $E%
\begin{array}
[c]{ccccc}%
\bullet &  &  &  & \\
\uparrow &  &  &  & \\
\bullet_{w} &  &  &  & \bullet_{u_{2}}\\
\downharpoonleft &  &  & \nearrow e_{1} & \\
\bullet_{v} & \longleftarrow & \bullet_{u_{1}} &  & \downarrow e_{2}\\
&  &  & \nwarrow e_{3} & \\
&  &  &  & \bullet_{u_{3}}%
\end{array}
$

Here $H=\{v\}$ is a hereditary saturated subset. Let $A=<H>$, the ideal
generated by $H$. Let $c$ denote the cycle $e_{1}e_{2}e_{3}$. Clearly $c$ has
no exits in $E\backslash H$. Let $B$ be the non-graded ideal $A+<p(c)>$, where
$p(x)=1+x\in K[x]$. Clearly $gr(B)=A$. Since $A$ is a graded ideal, we apply
Lemma \ref{graded Ideal Property} (a), to conclude that $AB=A\cap
B=A=A^{2}=AA$. But $A\neq B$.
\end{example}

\section{Prime, Radical, Primary and Irreducible ideals of a Leavitt path
algebra.}

In this section, we describe special types of ideals in $L$ such as the prime,
the irreducible, the primary and the radical (= semiprime) ideals using
graphical properties. While these concepts are independent for ideals in a
commutative ring, we show that the first three properties of ideals coincide
for graded ideals in the Leavitt path algebra $L$. We also show that a
non-graded ideal $I$ of $L$ is irreducible if and only if $I$ is a primary
ideal if and only if $I=P^{n}$, a power of a prime ideal $P$. This is useful
in the factorization of ideals in the next section. We also characterize the
radical ideals of $L$. It may be some interest to note that every graded ideal
of $L$ is a radical ideal.

The following description of prime ideals of $L$ was given in \cite{R1}.

\begin{theorem}
\label{R-1}(Theorem 3.2, \cite{R1}) \ An ideal $P$ of $L:=L_{K}(E)$ with
$P\cap E^{0}=H$ is a prime ideal if and only if $P$ satisfies one of the
following properties:

(i) $\ \ P=I(H,B_{H})$ and $E^{0}\backslash H$ is downward directed;\ 

(ii) $\ P=I(H,B_{H}\backslash\{u\})$, $v\geq u$ for all $v\in E^{0}\backslash
H$ and the vertex $u^{\prime}$ that corresponds to $u$ in $E\backslash
(H,B_{H}\backslash\{u\})$ is a sink;

(iii) $P$ is a non-graded ideal of the form $\ P=I(H,B_{H})+<p(c)>$, where $c
$ is a cycle without exits based at a vertex $u$ in $E\backslash(H,B_{H})$,
$v\geq u$ for all $v\in E^{0}\backslash H$ and $p(x)$ is an irreducible
polynomial in $K[x,x^{-1}]$ such that $p(c)\in P$.
\end{theorem}

Recall, an ideal $I$ of a ring $R$ is called an \textbf{irreducible ideal }if,
for ideals $A,B$ of $R$, $I=A\cap B$ implies that either $I=A$ or $I=B$. Given
an ideal $I$, the \textbf{radical of the ideal }$I$ , denoted by\textbf{\ }%
$Rad(I)$ or $\sqrt{I}$, is the intersection of all prime ideals containing
$I$. A useful property is that if $a\in Rad(I)$, then $a^{n}\in I$ for some
integer $n\geq0$ (The proof of this property is given in Theorem 10.7 of
\cite{Lam} for non-commutative rings with identity, but the proof also works
for rings without identity but with local units). If $Rad(I)=I$ for an ideal
$I$, then $I$ is called a \textbf{radical ideal }or\textbf{\ }a\textbf{
semiprime} \textbf{ideal}. An ideal $I$ of $R$ is said to be a \textbf{primary
ideal }if, for any two ideals $A,B$, if $AB\subseteq I$ and $A\nsubseteqq I$,
then $B\subseteq Rad(I)$.

\textbf{Remark}: We note in passing that for any graded ideal $I$ of $L$, say
$I=I(H,S)$, $Rad(I)=I$. Because, $Rad(I)/I$ is a nil ideal in $L/I$ and $L/I$,
being isomorphic to the Leavitt path algebra $L_{K}(E\backslash(H,S))$, has no
non-zero nil ideals.

We now point out an interesting property of graded ideals of $L$.

\begin{theorem}
(\cite{R3}) Suppose $I$ is a graded ideal of $L$. Then the following are equivalent:

(i) $\ \ I$ is a primary ideal;

(ii) $\ I$ is a prime ideal;

(iii) $I$ is an irreducible ideal.
\end{theorem}

The next theorem extends the above result to arbitrary ideals of $L$.

\begin{theorem}
\label{Irred = primary =prime power}(\cite{R3}) Suppose $I$ is a non-graded
ideal of $L$. Then the following are equivalent:

(i) $\ \ \ I$ is a primary ideal;

(ii) \ $\ I=P^{n}$, a power of a prime ideal $P$ for some $n\geq1$;

(iii) \ $I$ is an irreducible ideal.
\end{theorem}

The final result of this section describes the radical (also known as
semiprime) ideals of $L$.

\begin{theorem}
\label{semiprime ideals}(\cite{AMR}) Let $A$ be an arbitrary ideal of $L$ with
$A\cap E^{0}=H$ and $S=\{v\in B_{H}:v^{H}\in A\}$. Then the
following\ properties are equivalent:

(i) $\ A$ is a radical ideal of $L$;

(ii) $A=I(H,S)+%
{\displaystyle\sum\limits_{i\in Y}}
<f_{i}(c_{i})>$, where $Y$ is an index set which may be empty, for each $i\in
Y$, $c_{i}$ is a cycle without exits based at a vertex $v_{i}$ in
$E\backslash(H,S)$ and $f_{i}(x)$ is a polynomial with its constant term
non-zero which is a product of distinct irreducible polynomials in
$K[x,x^{-1}]$.
\end{theorem}

\section{Factorization of Ideals in \textit{L}}

As noted in the Introduction, ideals in an arithmetical ring admit interesting
representations as products of special types of ideals (\cite{F-3},
\cite{F-4}, \cite{F-5}). In this section, we explore the existence and the
uniqueness of factorizations of an arbitrary ideal in a Leavitt path algebra
$L$ as a product of prime ideals and as a product of irreducible/primary
ideals. The prime factorization of graded ideals of $L$ seems to influence
that of the non-graded ideals in $L$. Indeed, an ideal $I$ is a product of
prime ideals in $L$ if and only its graded part $gr(I)$ has the same property
and, moreover, $I/gr(I)$ is finitely generated with a generating set of
cardinality no more than the number of distinct prime ideals in an irredundant
factorization of $gr(I)$. It is interesting to note that if $I$ is a graded
ideal and if $I=P_{1}\cdot\cdot\cdot P_{n}$ is an irredundant product of prime
ideals, then necessarily each of the ideals $P_{j}$ must be graded. We also
show that $I$ is an intersection of irreducible ideals if and only if $I$ is
an intersection of prime ideals. If $L$ is the Leavitt path algebra of a
finite graph or, more generally, if $L$ is two-sided noetherian or two-sided
artinian, then every ideal of $L$ is shown to be a product of prime ideals. We
also give necessary and sufficient conditions under which every non-zero ideal
of $L$ is a product of prime ideals, that is when $L$ is a generalized ZPI
ring. We end this section by proving for $L$ an analogue of the Krull's
theorem on the intersection of powers of an ideal.

We begin with the following useful proposition.

\begin{proposition}
\label{gr part prime} (\cite{R3}) Suppose $I$ is a non-graded ideal of $L$. If
$gr(I)$ is a prime ideal. then $I$ is a product of prime ideals.
\end{proposition}

Using this, we obtain the following main factorization theorem.

\begin{theorem}
\label{gr(I) product of primes => I product of primes} (\cite{R3}) Let $E$ be
an arbitrary graph. For a non-graded ideal $I$ of $L:=L_{k}(E)$, the following
are equivalent:

(i) $\ \ \ I$ is a product of prime ideals;

(ii) $\ I$ is a product of primary ideals;

(iii) $I$ is a product of irreducible ideals;

(iv) $gr(I)$ is a product of (graded) prime ideals;

(v) $gr(I)=P_{1}\cap\cdot\cdot\cdot\cap P_{m}$ is an irredundant intersection
of $m$ graded prime ideals $P_{j}$ and $I/gr(I)$ is generated by at most $m$
elements and is of the form $I/gr(I)=%
{\displaystyle\bigoplus\limits_{r=1}^{k}}
<f_{r}(c_{r})>$ where $k\leq m$ and, for each $r=1\cdot\cdot\cdot k$, $c_{r}$
is a cycle without exits in $E^{0}\backslash I$ and $f_{r}(x)\in K[x]$ is a
polynomial of smallest degree such that $f_{r}(c_{r})\in I$.
\end{theorem}

As a consequence of Theorem
\ref{gr(I) product of primes => I product of primes}, we obtain a number of corollaries.

\begin{corollary}
(\cite{R3}) Let $E$ be a finite graph, or more generally, let $E^{0}$ be
finite. Then every non-zero ideal of $L=L_{K}(E)$ is a product of prime ideals.
\end{corollary}

Using a minimal or maximal argument, the above corollary can be extended to
the case when the ideals of $L$ satisfy the DCC or ACC as noted below.

\begin{corollary}
(\cite{R3}) Suppose $L$ is two-sided artinian or two-sided noetherian. Then
every non-zero ideal of $L$ is a product of prime ideals.
\end{corollary}

We now give the necessary and sufficient conditions under which $L$ is a
generalized ZPI ring, that is when every ideal of $L$ is a product of prime ideals.

\begin{theorem}
(\cite{R3}) Let $E$ be an arbitrary graph and let $L:=L_{K}(E)$. Then every
proper ideal of $L$ is a product of prime ideals if and only if every
homomorphic image of $L$ is either a prime ring or contains only finitely many
minimal prime ideals.
\end{theorem}

The next theorem states that an irredundant factorization of an ideal $A$ as a
product of prime ideals in $L$ is unique up to a permutation of the factors.
It also points out the interesting fact that if $A$ is a graded ideal, then
every factor in this irredundant factorization must also be a graded ideal.

Recall that $A=P_{1}\cdot\cdot\cdot P_{n}$ is an \textbf{irredundant product}
of the ideals $P_{i}$, if $A$ is not the product of a proper subset of the set
$\{P_{1},\cdot\cdot\cdot,P_{n}\}$.

\begin{theorem}
(\cite{EER}) (a) Suppose $A$ is an arbitrary ideal of $L$ and $A=P_{1}%
\cdot\cdot\cdot P_{m}=Q_{1}\cdot\cdot\cdot Q_{n}$ are two representations of
$A$ as irredundant products of prime ideals $P_{i}$and $Q_{j}$. Then $m=n$ and
$\{P_{1},\cdot\cdot\cdot,P_{m}\}=\{Q_{1},\cdot\cdot\cdot,Q_{n}\}$;

(b) If $A$ is a graded ideal of $L$ and if $A=P_{1}\cdot\cdot\cdot P_{m}$ is
an irredundant product of prime ideals $P_{j}$, then the ideals are all graded
and $A=P_{1}\cap\cdot\cdot\cdot\cap P_{m}$.
\end{theorem}

From Proposition \ref{graded Ideal Property}(c), and the equivalence of
conditions (i) and (iv) of Theorem
\ref{gr(I) product of primes => I product of primes}, we derive following Proposition.

\begin{proposition}
If an ideal $I$ of $L$ is an intersection of finitely many prime ideals, then
$I$ is a product of (finitely many) prime ideals.
\end{proposition}

But a product of prime ideals in $L$ need not be an intersection of prime
ideals as the next example shows.

\begin{example}
If $E$ is the graph with a single vertex and a single loop
\[%
\raisebox{-0pt}{\includegraphics[
natheight=0.396100in,
natwidth=0.989300in,
height=0.4255in,
width=1.0222in
]%
{Loop.jpg}%
}
\]
then $L_{K}(E)\cong K[x,x^{-1}]$, the Laurent polynomial ring, induced by the
map $v\mapsto1$, $x\mapsto x$, $x^{\ast}\mapsto x^{-1}$. So it is enough to
find a ideal $A$ in $K[x,x^{-1}]$ with the desired property. Consider the
prime ideal $A=<p(x)>$ in $K[x,x^{-1}]$, where $p(x)$ is an irreducible
polynomial. We claim that $B=A^{2}$ is not an intersection of prime ideals in
$K[x,x^{-1}]$. Suppose, on the contrary, $B=%
{\displaystyle\bigcap\limits_{\lambda\in\Lambda}}
M_{\lambda}$ where $\Lambda$ is some (finite or infinite) index set and each
$M_{\lambda}$ is a (non-zero) prime ideal of $K[x,x^{-1}]$ and hence a maximal
ideal of the principal ideal domain $K[x,x^{-1}]$. Now there is a homomorphism
$\phi:R\longrightarrow%
{\displaystyle\prod\limits_{\lambda\in\Lambda}}
R/M_{\lambda}$ given by $r\mapsto(\cdot\cdot\cdot,r+M_{\lambda},\cdot
\cdot\cdot)$ with $\ker(\phi)=B$. Then $\bar{A}=\phi(A)\cong A/B\neq0$
satisfies $(\bar{A})^{2}=0$ and this is impossible since $%
{\displaystyle\prod\limits_{\lambda\in\Lambda}}
R/M_{\lambda}$, being a direct product of fields, does not contain any
non-zero nilpotent ideals.
\end{example}

The next Proposition is new and gives necessary and sufficient conditions
under which a product of prime ideals in a Leavitt path algebra is also an
intersection of prime ideals. This happens exactly when every ideal of $L$ is
a radical ideal.

\begin{proposition}
Let $E$ be an arbitrary graph and let $L:=L_{K}(E)$. Then the following
properties are equivalent:

(i) \ \ Every product of prime ideals in $L$ is an intersection of prime ideals;

(ii) \ The graph $E$ satisfies Condition (K);

(iii) Every ideal of $L$ is a radical ideal.
\end{proposition}

\begin{proof}
Assume (i). Assume, by way of contradiction that the graph $E$ does not
satisfy Condition (K). Then, for some admissible pair $(H,S)$, the quotient
graph $E\backslash(H,S)$ does not satisfy Condition (L) (see \cite{AAS}) and
thus there is a cycle $c$ without exits in $E\backslash(H,S)$. By [\cite{AAS},
Lemma 2.7.1], the ideal $M$ of $L_{K}(E\backslash(H,S))$ generated by
$\{c^{0}\}$ is isomorphic to the matrix ring $M_{\Lambda}(K[x,x^{-1}])$  where
$\U{39b} $ some index set. Then [\cite{EEKR}, Proposition 1] and Example 4
above implies that, for any prime ideal $P$ of $M$, $P^{2}$ is not an
intersection of prime ideals of $M$. Since the graded ideal $M$ is a ring with
local units ([\cite{AAS}, Corollary 2.5.23]), every ideal (prime ideal) of $M$
is an ideal (prime ideal) of $L_{K}(E\backslash(H,S))$ and, for any prime
ideal $Q$ of $L_{K}(E\backslash(H,S))$, $M\cap Q$ is a prime ideal of $M$.
Consequently, $P^{2}$ be an intersection of prime ideals of $L_{K}%
(E\backslash(H,S))$. This is a contradiction, since $L_{K}(E\backslash(H,S))$,
being isomorphic to the quotient ring $L/I(H,S)$, satisfies (i). Consequently,
the graph $E$ must satisfy Condition (K), thus proving (ii).

Assume (ii). By [\cite{AAS}, Proposition 2.9.9], every ideal of $L$ is graded.
On the other hand if $I=I(H,S)$ is a graded ideal, then$L/I$ is isomorphic to
the Leavitt path algebra $L_{K}(E\backslash(H,S))$ (\cite{AAS}) \ and since
the prime radical (the intersection of all primes ideals of $L_{K}%
(E\backslash(H,S))$ is zero, $I$ is the intersection of all the prime ideals
containing $I$ and hence is a radical ideal. This proves (iii).

Assume (iii). We claim that every ideal of $L$ must be a graded ideal.
Suppose, by way of contradiction, there is a non-graded ideal $I$ in $L$, say,
$I=I(H,S)+%
{\displaystyle\sum\limits_{i\in Y}}
f_{i}(c_{i})$, where $Y$ is an index set and, for each $i\in Y$, $f_{i}(x)\in
K[x]$ and $c_{i}$ is a cycle without exits in $E\backslash(H,S)$. Now for a
fixed $i\in Y$ and an irreducible polynomial $p(x)\in K[x,x^{-1}]$,
$P=I(H,S)+<p(c_{i})>$  is a prime ideal and $\bar{P}=P/I(H,S)=<p(c_{i}%
)>\subseteqq M=<\{c_{i}^{0}\}>$. As noted in the proof of (i) =\U{21d2} (ii),
$P^{2}$ is not a radical ideal of $L/I(H,S)$ and hence $P^{2}$ is not a
radical ideal in $L$, a contradiction. Hence every ideal of $L$ is a graded
ideal. This proves (iv).

Now (iv) =\U{21d2} (i), by Proposition \ref{graded Ideal Property}(c).
\end{proof}

We end this section by considering the powers of an ideal in $L$. From
Proposition \ref{graded Ideal Property}, it is clear that if $A$ is a graded
ideal of $L$, then $A=A^{2}$ and so $A=A^{n}$ for all $n\geq1$. What happens
if $A$ is a non-graded ideal ? The next Proposition implies that, for such an
$A$, $A\neq A^{n}$ for any $n>1$.

\begin{proposition}
\label{Intersection of powers} (\cite{EER}) If $A$ is a non-graded ideal in
$L$, then $%
{\displaystyle\bigcap\limits_{n=1}^{\infty}}
A^{n}$ is a graded ideal, being equal to $gr(A)$.
\end{proposition}

As a corollary we obtain,

\begin{corollary}
An ideal $A$ of $L$ is a graded ideal if and only if $A=$ $A^{n}$ for all
$n\geq1$.
\end{corollary}

W. Krull showed that if $A$ is an ideal of a commutative noetherian ring with
identity $1$, then $%
{\displaystyle\bigcap\limits_{n=1}^{\infty}}
A^{n}=0$ if and only if $1-x$ is not a zero divisor for all $x\in A$ (see
Theorem 12, Section 7 in \cite{ZS}). As an consequence of Proposition
\ref{Intersection of powers}, we obtain an analogue of Krull's theorem for
Leavitt path algebras.

\begin{corollary}
(\cite{EER}) Let $A$ be an arbitrary ideal of $L$. Then $%
{\displaystyle\bigcap\limits_{n=1}^{\infty}}
A^{n}=0$ if and only if\ $A$ contains no vertices of the graph $E$.
\end{corollary}

\begin{acknowledgement}
My thanks to Gene Abrams for carefully reading this article, making
corrections and offering suggestions.
\end{acknowledgement}

\end{document}